\documentclass[12pt]{amsart}
\usepackage{amssymb}
\usepackage{t1enc}
\usepackage[mathscr]{eucal}
\usepackage{parskip}
\usepackage{xcolor}
\usepackage{tikz}
\usepackage{tikz-cd}
\usepackage{hyperref}
\usepackage[capitalize]{cleveref}
\numberwithin{equation}{section}
\usepackage[margin=2.9cm]{geometry}
\usepackage{booktabs}
\usepackage{enumerate}
\usepackage{bm}
\usepackage{xfrac}
\usepackage{mathtools}
\usepackage{nicematrix}

\usepackage{listings}
\theoremstyle{plain}
\newtheorem{thm}[equation]{Theorem}
\newtheorem{lem}[equation]{Lemma}
\newtheorem{cor}[equation]{Corollary}
\newtheorem{prop}[equation]{Proposition}

 \theoremstyle{definition}
\newtheorem{defn}[equation]{Definition}
\newtheorem{rem}[equation]{Remark}

\newtheorem{notn}[equation]{Notation}

\newtheorem{ques}[equation]{Question}

\newtheorem{constr}[equation]{Construction}

\crefname{thm}{Theorem}{Theorems}
\crefname{lem}{Lemma}{Lemmas}
\crefname{cor}{Corollary}{Corollaries}
\crefname{prop}{Proposition}{Propositions}
\crefname{conj}{Conjecture}{Conjectures}
\crefname{defn}{Definition}{Definitions}
\crefname{rem}{Remark}{Remarks}
\crefname{ex}{Example}{Examples}
\crefname{constr}{Construction}{Constructions}

\newcommand{\mb}[1]{\mathbb{#1}}
\newcommand{\mc}[1]{\mathcal{#1}}

\newcommand{\rank}{\operatorname{rank}}

\newcommand{\Kt}{\mathrm{K3}}
\newcommand{\U}{\mathrm{U}}
\newcommand{\SU}{\mathrm{
SU}}
\newcommand{\Sp}{\mathrm{Sp}}
\newcommand{\Spin}{\mathrm{Spin}}
\newcommand{\SO}{\mathrm{SO}}
\renewcommand{\O}{\mathrm{O}}
\newcommand{\GL}{\mathrm{GL}}
\newcommand{\SL}{\mathrm{SL}}
\newcommand{\KU}{\mathrm{KU}}
\newcommand{\KO}{\mathrm{KO}}
\newcommand{\KSp}{\mathrm{KSp}}
\newcommand{\M}{\mathrm{M}}
\newcommand{\B}{\mathrm{B}}
\newcommand{\Ext}{\operatorname{Ext}}
\newcommand{\Hom}{\operatorname{Hom}}
\newcommand{\even}{\mathrm{even}}
\newcommand{\ch}{\mathrm{ch}}
\newcommand{\po}{\mathrm{po}}
\newcommand{\sy}{\mathrm{sy}}
\newcommand{\cd}{\mathrm{cd}}
\renewcommand{\c}{\check{c}}
\newcommand{\id}{\mathrm{id}}

\def\colim{\qopname\relax m{colim}}

\makeatletter
\@namedef{subjclassname@2020}{\textup{2020} Mathematics Subject Classification}
\makeatother

\begin{document}
\title{Hyperkähler bases for six rational bordism theories}

\author[Buchanan]{Jonathan Buchanan}
\address{Department of Mathematics \\ Massachusetts Institute of Technology}
\email{jbuch333@mit.edu}
\urladdr{https://jonathanbuchanan.net/}

\author[Debray]{Arun Debray}
\address{Department of Mathematics \\ University of Kentucky}
\email{a.debray@gmail.com}
\urladdr{https://adebray.github.io/}

\author[Krulewski]{Cameron Krulewski}
\address{Department of Mathematics and Statistics\\ Dalhousie University}
\email{ckrulewski@dal.ca}
\urladdr{https://cakrulewski.github.io/}

\author[McKean]{Stephen McKean}
\address{Department of Mathematics \\ Brigham Young University} 
\email{mckean@math.byu.edu}
\urladdr{https://shmckean.github.io/}


\begin{abstract}
We use tori and Hilbert schemes of K3 surfaces to construct explicit bases for the real, complex, and quaternionic versions of rational symplectic and rational Spin bordism. The key input to our work is a theorem of Oberdieck, Song, and Voisin on the Milnor genus of Hilbert schemes of K3s.
\end{abstract}

\maketitle

\section{Introduction}
Symplectic\footnote{See \cref{rem:symplectic}.} bordism is complicated. Thanks to hard work by many mathematicians including
Novikov~\cite{novikov_homotopy_2007},
Liulevicius~\cite{Liu64},
Stong~\cite{Sto67},
Ray~\cite{Ray70, Ray71, Ray71a},
Segal~\cite{Seg70},
Kochman~\cite{Koc80, Koc82, Koc93},
Vershinin~\cite{Ver80, Ver83},
Botvinnik~\cite{Bot89, Bot92},
Botvinnik--Kochman~\cite{BK94, BK96},
and
Anisi\-mov--Ver\-shi\-nin~\cite{AV12},
the symplectic bordism groups have been computed through degree $100$, but a general description is not known. The situation is much simpler rationally, where the bordism ring is isomorphic to an infinite polynomial ring with generators given by the dual symplectic Pontryagin classes. From this, one can count ranks to show that rational symplectic and rational oriented bordism are isomorphic.

Thanks to foundational work of Anderson, Brown, and Peterson \cite{ABP67}, we have a complete description of Spin bordism groups (although the ring structure is still not completely known). Rationally, the Spin bordism ring is isomorphic to the oriented bordism ring, again by counting characteristic classes. 

Thom proved that rational oriented bordism ring is generated by the projective spaces $\mb{CP}^{2n}$ \cite[Corollaire IV.18]{ThomThesis}. Despite the aforementioned isomorphisms, $\mb{CP}^{2n}$ admits neither almost quaternionic nor Spin structure, so this basis for $\Omega^{\SO}_*\otimes\mb{Q}$ does not induce a basis for $\Omega^{\Sp}_*\otimes\mb{Q}$ or $\Omega^{\Spin}_*\otimes\mb{Q}$.

In this article, we construct manifolds that generate rational Sp and Spin bordism, along with their complex and quaternionic counterparts. In contrast to previously known generating sets (see \S\ref{sec:context}), our manifolds are hyperkähler (in dimensions $4n$) or are the product of a hyperkähler manifold with a 2-torus (in dimensions $4n+2$). We will state our main theorem after introducing some notation.

Given two groups $G,K$ with chosen central subgroups $\{\pm 1\}\subseteq G$ and $\{\pm 1\}\subset K$, let $G\cdot K$ denote $G\times K$ modulo the diagonal subgroup $(\pm 1, \pm 1)$. With this notation, set
\begin{equation}
\label{spin_quotients}
\begin{alignedat}{3}
    \Sp^r(n)&\coloneqq\Sp(n), && & \Spin^r(n)&\coloneqq\Spin(n),\\
    \Sp^c(n)&\coloneqq\Sp(n)\times\U(1), && \qquad\qquad & \Spin^c(n)&\coloneqq\Spin(n)\cdot\U(1),\\
    \Sp^h(n)&\coloneqq\Sp(n)\times\SU(2), && & \Spin^h(n)&\coloneqq\Spin(n)\cdot\SU(2),
\end{alignedat}
\end{equation}
where $\Sp$ and $\Spin$ denote the symplectic and spin groups, respectively. Here, $\{ \pm 1 \} \subseteq \Spin(n)$ is the kernel of the double cover $\Spin(n) \to \SO(n)$, and $\{ \pm 1 \}$ is the subgroup generated by the matrix $-I$ in $\U(1)$ and $\SU(2)$. For $x\in\{r,c,h\}$, let $\displaystyle\Sp^x\coloneqq\colim_{n\to\infty}\Sp^x(n)$ and $\displaystyle\Spin^x\coloneqq\colim_{n\to\infty}\Spin^x(n)$. (See \S\ref{sec:justification} for an explanation of the lack of parallelism between $\Sp^x$ and $\Spin^x$.)

In \S\ref{sec:bordism isomorphisms}, we describe a group homomorphism $\widetilde f\colon \Sp\to\Spin$, which we call a ``forgetful map.''\footnote{\label{attrfoot}This map, and the induced map on bordism, are both well-known but not necessarily clearly spelled out in the literature. For example, one can factor $f$ as a map $\Sp\to\SU$, which was described by Conner--Floyd~\cite[\S 5]{CF66}, followed by a map $\SU\to\Spin$, discussed in e.g.\ Stong~\cite[Chapter XI]{Sto68}.} Using it, we can define forgetful maps $\widetilde f^x\colon \Sp^x\to\Spin^x$ for $x \in\{c,h\}$ to be the compositions
\begin{equation}
\label{CH_forget}
    \widetilde f^x\colon \Sp^x = \Sp\times G^x \xrightarrow{(f,\, \mathrm{id})}
        \Spin\times G^x\overset{q}{\longrightarrow}
        \Spin\cdot G^x,
\end{equation}
where $G^c = \U(1)$, $G^h = \SU(2)$, and $q$ is the quotient by the diagonal subgroup $(\pm 1,\pm 1)\coloneqq\{(1,1),(-1,-1)\}$.

Given a complex algebraic surface $S$ and $n\geq 1$, let $S^{[n]}$ denote the Hilbert scheme of $n$ points on $S$. Fogarty proved that if $S$ is smooth and projective, then so is $S^{[n]}$ \cite{Fog68}. Beauville proved that if $S$ is also hyperkähler, then $S^{[n]}$ is hyperkähler \cite[Théorème 3]{Bea83}. In particular, Hilbert schemes of points on K3 surfaces are always smooth and admit an $\Sp$-structure. 

Let $\mc{P}(n)$ denote the set of partitions of $n$. Let $\Kt$ denote any complex K3 surface (all of which belong to the same diffeomorphism class). Let $T^n$ denote the $n$-torus, which is complex in even dimensions $n$. 

\begin{notn}\label{notn:bundles}
For $x=c$, equip $T^{2n}$ with a framing, which exists because $T^{2n}$ is a Lie group. Then consider the line bundle
\begin{equation}
\label{torus_line}
    L_{2n}\to T^{2n},
\end{equation}
which is defined as the $n\textsuperscript{th}$ external power of a complex line bundle $L\to T^2$ with Chern number 1. We denote the torus with this structure by $T^{2n}_c$.

For $x=h$, we keep the framing and define a principal $\SU(2)$-bundle as follows. For any choice of metric, the unit sphere bundle inside the line bundle $L_{4n}$ is a $\U(1)$-bundle $S_{4n}\to T^{4n}$, whose isomorphism type is independent of the choice of metric. The associated $\SU(2)$-bundle (see~\S\ref{conventions}) is
\begin{equation}
\label{torus_SU2}
    Q_{4n}\coloneqq S_{4n}\times_{\U(1)}\SU(2)\to T^{4n}.
\end{equation}
We denote the torus with this structure by $T^{4n}_h$.
\end{notn}

\begin{thm}\label{thm:main}
    The forgetful map induces an isomorphism 
    \begin{equation}\label{eq:iso}
    \widetilde F{}^x\colon \Omega^{\Sp^x}_*\otimes\mb{Q} \xrightarrow{\cong} \Omega^{\Spin^x}_*\otimes\mb{Q}
    \end{equation}
    of rings for $x\in\{r,c\}$ and of $\Omega^{\Sp}_*\otimes\mb{Q}$-modules for $x=h$. Moreover, the (non-trivial) graded pieces of Equation~\eqref{eq:iso} admit the following bases $B^x_*$:
    \begin{align*}
        B^r_{4n}&=\left\{\prod_{i=1}^a\Kt^{[n_i]}:(n_1,\ldots,n_a)\in\mc{P}(n)\right\},\\
        B^c_{4n}&=\left\{T^{4(n-m)}_c\times\prod_{i=1}^a\Kt^{[m_i]}:0\leq m\leq n\text{ and }(m_1,\ldots,m_a)\in\mc{P}(m)\right\},\\
        B^c_{4n+2}&=\left\{T^{4(n-m)+2}_c\times\prod_{i=1}^a\Kt^{[m_i]}:0\leq m\leq n\text{ and }(m_1,\ldots,m_a)\in\mc{P}(m)\right\},\\
        B^h_{4n}&=\left\{T^{4(n-m)}_h\times\prod_{i=1}^a\Kt^{[m_i]}:0\leq m\leq n\text{ and }(m_1,\ldots,m_a)\in\mc{P}(m)\right\}.
    \end{align*}
\end{thm}
By~\cite[Theorem 0.1]{EGL01}, the bordism classes of these basis elements do not depend on the choice of K3 surface.

Our proof of \cref{thm:main} consists of the following steps.
\begin{enumerate}[(i)]
\item Put both a $\Sp^x$ and a $\Spin^x$ structure on the elements of $B^x_*$. This is done in \cref{sec:preliminaries} and relies on an observation of Beauville \cite{Bea83}.
\item Construct the maps  $\tilde{F}^x:\Omega^{\Sp^x}_n\otimes\mb{Q}\to\Omega^{\Spin^x}_n\otimes\mb{Q}$ (see \eqref{eq:iso}), prove that these are isomorphisms for all $n$, and deduce that $|B^x_*|=\rank\Omega^{\Sp^x}_*$. This is done in \cref{sec:bordism isomorphisms} and relies only on standard techniques from homotopy theory.
\item Prove that the elements of $B^x_*$ are linearly independent in $\Omega^{\Sp^x}_*\otimes\mb{Q}$. This is done in \cref{sec:bases} and relies on work of Oberdieck, Song, and Voisin \cite{OSV22}.
\end{enumerate}

We do not know how far $B^x_*$ are from being bases of $\Omega^{G^x}_*/(\text{torsion})$ or $\Omega^{G^x}_*\otimes\mb{Z}[1/2]$, where $G\in\{\Sp,\Spin\}$.

\begin{ques}
    For $G\in\{\Sp,\Spin\}$ and $x\in\{r,c,h\}$, does the set $B^x_*$ generate either $\Omega^{G^x}_*/(\text{torsion})$ or $\Omega^{G^x}_*\otimes\mb{Z}[1/2]$?
\end{ques}

\begin{rem}
    Technically, the title of this article is a lie. The elements of $B^c_{4n+2}$ are not hyperkähler, because $T^2$ is not hyperkähler. A more cumbersome but accurate title would be ``Hyperkähler bases for four rational bordism theories, and a near miss for two more.''
\end{rem}

\subsection{Context}\label{sec:context}
For bordism theories with sufficiently simple structure group, the bordism groups are characterized in terms of certain characteristic classes via the Pontryagin--Thom theorem. For example, the unoriented bordism ring is isomorphic to $\mb{F}_2[x_i:i\neq 2^r-1]$, where $x_i$ is dual to a product of Stiefel--Whitney classes; $x_i$ is represented by $\mb{RP}^i$ (for even $i$) or a Dold manifold~\cite{Dol56} (for odd $i$).\footnote{See~\cite{Mil65a, And66, OR77, Roy77, And87, Gsc96} for additional constructions of generating sets of manifolds for the unoriented bordism ring.}
The oriented bordism ring is generated over the integers by duals of Pontryagin classes and Stiefel--Whitney classes, which makes constructing manifold representatives more complicated, though see~\cite{Wal60, And66, Fuh22} for some sets of generating manifolds. The situation is simpler after rationalizing, in which case the oriented bordism ring is just generated by duals of Pontryagin classes, represented by $\mb{CP}^{2i}$~\cite[Corollaire IV.18]{ThomThesis}.

Though $\Sp^c$ and $\Sp^h$ bordism have not been studied before to our knowledge, there is prior work finding manifolds generating the other four bordism rings and modules that we study, sometimes after tensoring with $\mb{Z}[1/2]$ or $\mb{Q}$.
\begin{description}
    \item[Spin] Stong~\cite{Sto66} gave a generating set for $\Omega_*^\Spin\otimes\mb{Z}[1/2]$. Milnor~\cite{Mil63} found some low-dimensional generators over $\mb{Z}$. Stolz~\cite[Theorem B]{Sto92} showed that over $\mb{Z}$, most generators can be taken as total spaces of $\mb{HP}^2$-bundles. While Anderson--Brown--Peterson~\cite[p.~258]{ABP66} provides a semi-explicit description of manifold representatives for generators of rational Spin bordism in each degree, these representatives are in turn defined in terms of manifolds that are known to exist but are not explicitly described in most cases.
    \item[Spin\textsuperscript{\itshape c}] Stong~\cite{Sto66} gave a generating set over $\mb{Z}[1/2]$. Granath~\cite{Gra23} showed that, over $\mb{Z}$, the generators may be chosen to be the point and the total spaces of $\mb{HP}^2$- and $\mb{CP}^2$-bundles. See also Abdallah--Salch~\cite{AS25} for recent work on the multiplicative structure over $\mb{Z}$.
    \item[Spin\textsuperscript{\itshape h}] Hu~\cite[Appendix A]{Hu23} finds generators over $\mb{Z}$ in degrees $\le 5$; Buchanan--McKean~\cite[\S 10]{BM23} also study this question. Debray--Krulewski~\cite[Theorem 4.13]{DK25} show that over $\mb{Z}[1/2]$, a set of generators for $\Omega_{4*}^{\Spin^c}$, such as Stong's, induces a set of generators of $\Omega_{4*}^{\Spin^h}$.
    \item[Sp] Stong~\cite{Sto67} gave a generating set over $\mb{Z}[1/2]$. Ray~\cite{Ray72, Ray86} and Laughton~\cite[\S 7]{Lau08} give partial results over $\mb{Z}$; see also Gigli~\cite[\S 6]{Gig25} for work towards generators of a motivic analogue of symplectic bordism.
\end{description}
Our goal of building a set of generators with nice geometric properties also has antecedents, including Oberdieck--Song--Voisin's work~\cite{OSV22} as mentioned above, as well as that of Limonchenko--Lü--Panov~\cite{LLP18}, who construct generators for $\Omega_*^\mathrm{SU}\otimes\mb{Z}[1/2]$ that are Calabi--Yau hypersurfaces in toric varieties.

\subsection{Conventions}\label{conventions}
If $M$ is a manifold and $k$ is a characteristic class, then we use the notation $k(M)\coloneqq k(TM)$.

We choose an inclusion of $\mb{R}$-algebras $\mb{C}\hookrightarrow\mb{H}$: unlike for $\mb{R}\hookrightarrow\mb{C}$, there is not a unique choice.
We select the unique algebra homomorphism $\mb{C}\hookrightarrow\mb{H}$ obtained by extending $\mb{R}$-linearly from $1\mapsto 1$ and $i\mapsto i$. 
None of our constructions depend on this specific choice, but it is helpful to make one.

From this homomorphism, we obtain forgetful functors from 
left $\mb{H}$-modules to complex vector spaces and from manifolds with an $\Sp$-structure to almost complex manifolds. We also obtain injective group homomorphisms $\Sp(n)\hookrightarrow\U(2n)$ for all $n$ (and can then forget further to $\O(4n)$). Restricting $\mb{C}\hookrightarrow\mb{H}$, we also obtain an injective group homomorphism $\U(1)\hookrightarrow\SU(2)$, which is the standard maximal torus of $\SU(2)$.

\begin{rem}\label{rem:symplectic}
    There is an unfortunate conflation of the term \emph{symplectic} in the literature: there are two distinct types of groups known as symplectic groups, and hence two distinct notions of symplectic structures on manifolds. Confusingly, \emph{symplectic} for us will mean the compact symplectic group, which is arguably the more obscure variant of the two.
    
    The \emph{symplectic group} $\Sp(2n,\mb{R})$ is the subgroup of $\SL(2n,\mb{R})$ of linear transformations preserving a non-degenerate, skew-symmetric bilinear form, while the \emph{compact symplectic group} $\Sp(n)$ is the subgroup of $\GL(n,\mb{H})$ of linear transformations preserving the standard hermitian form.

    A $2n$-manifold $M$ is said to be \emph{almost symplectic} if it admits a non-degenerate $2$-form $\omega$. The structure group for this sort of symplectic structure is $\Sp(2n,\mb{R})$. If $\omega$ is also closed, then $\omega$ is called a \emph{symplectic form} and $M$ is said to be a \emph{symplectic manifold}. 
    
    On the other hand, a manifold $M$ having structure group contained in $\Sp(n)$ corresponds to the stable normal (or tangent) bundle of $M$ admitting the structure of a quaternionic vector bundle. In analogy with \emph{almost complex} structures, it would seem reasonable to use \emph{almost quaternionic} structure as a pithy replacement for $\Sp$-structure; unfortunately, almost quaternionic already means something slightly different.
    
    A Riemannian $4n$-manifold with holonomy contained in $\Sp(n)$ is necessarily a hyperkähler manifold \cite[p.~292]{Cal79}. Hyperkähler manifolds are \emph{holomorphically} symplectic (meaning that they admit a symplectic form that is holomorphic), and conversely, a compact kähler manifold with holomorphic symplectic form is hyperkähler \cite[p.~758]{Bea83}. Still, not every symplectic manifold is hyperkähler, since a holomorphic symplectic form trivializes the canonical bundle of the manifold. Consequently, $\mb{CP}^n$ is symplectic but not hyperkähler.

    The clash of the symplectics even extends to bordism, as there are well-known constructions in symplectic geometry known as symplectic cobordisms (see e.g.~\cite{EH02,EM23}). We will never mean this sort of cobordism when we say symplectic bordism.

    So to be clear, outside of this remark, we only work with the compact symplectic group. In particular, \emph{symplectic} means that the stable normal bundle admits the structure of a quaternionic vector bundle, not the existence of a symplectic form on our manifold.
\end{rem}

\subsection*{Acknowledgments}
JB is supported by an MIT Dean of Science Fellowship. CK is supported by the Killam Trusts and the Simons Foundation. SM was partially supported by the NSF (DMS-2502365) and the Simons Foundation.

\section{Characteristic classes}
\label{charclass}
We recall a few well-known facts about characteristic classes for our convenience. Perhaps the quickest way to define the classes we need are in terms of the cohomology of classifying spaces. To begin, we have an isomorphism
\begin{equation}
    H^*(\B\U;\mb{Z})\cong\mb{Z}[c_1,c_2,\ldots],
\end{equation}
where $|c_i|=2i$ \cite[p.~183]{Bor53}. The generators $c_i$ are known as \emph{Chern classes}. Given a real vector bundle $E$, we can then define
\begin{equation}
    p_i(E)\coloneqq(-1)^ic_{2i}(E\otimes_\mb{R}\mb{C}).
\end{equation}
This yields the \emph{Pontryagin classes} $p_1,p_2,\ldots\in H^*(\B\O;\mb{Z})$, with $|p_i|=4i$. Although the Pontryagin classes do not generate the integral cohomology of $\B\O$, there is an isomorphism~\cite[\S 30.2]{BH2}
\begin{equation}
    H^*(\B\O;\mb{Q})\cong\mb{Q}[p_1,p_2,\ldots].
\end{equation}
Since a complex bundle $E$ satisfies
\begin{equation}\label{eq:complexify}
E\otimes_\mb{R}\mb{C}=E\oplus\overline{E},
\end{equation}
the Whitney sum formula implies that
\begin{equation}
    \sum_{i=0}^\infty(-1)^ip_i(E) =\left(\sum_{i=0}^\infty c_i(E)\right)\left(\sum_{i=0}^\infty(-1)^i c_i(E)\right).
\end{equation}
Expanding this out, we find that
\begin{equation}\label{eq:chern to pontryagin}
p_k(E) = c_k(E)^2+2\sum_{i=0}^{k-1}(-1)^{k+i}c_{2k-i}(E)c_i(E).
\end{equation}

\begin{rem}
The relationship between Chern and Pontryagin classes is nicely packaged in terms of the Chern and Pontryagin characters. These are rational classes
\begin{equation}
    \ch\in\prod_{n=0}^\infty H^n(\B\U;\mb{Q}),\qquad \po\in\prod_{n=0}^\infty H^n(\B\O;\mb{Q})
\end{equation}
that represent ring homomorphisms from $\KU$-cohomology and $\KO$-cohomology, respectively. 
These homomorphisms are in turn induced by the Chern--Dold characters
\begin{equation}
    \cd_E\colon E\to E_\mb{Q},
\end{equation}
where $E=\KU$ or $\KO$ and $E_\mb{Q}$ denotes the rationalization of $E$. This relies on the fact that
\begin{equation}
\begin{aligned}
    \KU_\mb{Q}&\simeq H\mb{Q}[\![c_1]\!],\\
    \KO_\mb{Q}&\simeq H\mb{Q}[\![p_1]\!].
\end{aligned}
\end{equation}
See \cite[\S 2.1]{LSW20} for some nice exposition on the story of Chern--Dold characters.

The graded pieces of the characters $\ch$ and $\po$ can be computed by the Newton--Girard polynomials $s_k(x_1,\ldots,x_k)$, which are characterized by their generating series
\begin{equation}
    \sum_{k=1}^\infty(-1)^{k-1}s_k(x_1,\ldots,x_k)\frac{t^k}{k}=\log\left(1+\sum_{i=1}^\infty x_it^i\right).
\end{equation}
The first few examples are
\begin{equation}
\begin{aligned}
    s_1 &= x_1,\\
    s_2 &= x_1^2-2x_2,\\
    s_3 &= x_1^3-3x_1x_2+3x_3,\\
    s_4 &= x_1^4-4x_1^2x_2+4x_1x_3+2x_2^2-4x_4,\\
    s_5 &= x_1^5-5x_2x_1^3+5x_3x_1^2+5x_2^2x_1-5x_4x_1-5x_3x_2+5x_5.
\end{aligned}
\end{equation}
With this notation, we have
\begin{subequations}
\begin{align}
    \ch&=\sum_{k\geq 1}\frac{s_k(c_1,\ldots,c_k)}{k!},\\
    \po&=\sum_{k\geq 1}\frac{s_k(p_1,\ldots,p_k)}{k!}.
\end{align}
\end{subequations}
By definition of Pontryagin classes, the Chern and Pontryagin characters satisfy
\begin{equation}
    \po=\ch(-\otimes_\mb{R}\mb{C}).
\end{equation}
In terms of the graded pieces of these characters, we have $\po_k=\ch_{2k}(-\otimes_\mb{R}\mb{C})$. In particular, for all $k\geq 1$, we have
\begin{equation}
    s_k(p_1,\ldots,p_k)=s_{2k}(c_1,\ldots,c_{2k}),
\end{equation}
which is amusing to verify directly in terms of Equation~\eqref{eq:chern to pontryagin} and the explicit formulas for $s_k$ (c.f.~\cite[Problem 16-C]{MS74}).
\end{rem}

We will need the following characteristic class computation.

\begin{lem}
\label{T2y_charclass}
\[\int_{T^{2n}} c_1(L_{2n})^n = n!\]
\end{lem}
\begin{proof}
Recall (from \cref{notn:bundles}) that $L_{2n}$ is the external tensor product of the $n$ factors of $L\to T^2$. The first Chern class is additive in tensor product, so in $H^*(T^{2n};\mb{Z})\cong\mb{Z}[e_1, f_1,\dotsc,e_n, f_n]/(e_i^2, f_i^2)$, we have
\begin{subequations}
\begin{equation}
    c_1(L_{2n}) = \sum_{\ell=1}^n e_\ell f_\ell.
\end{equation}
By the binomial theorem, we have
\begin{equation}
    c_1(L_{2n})^n = n!\prod_{\ell=1}^n e_\ell f_\ell.
\end{equation}
\end{subequations}
The generator of $H^{2n}(T^{2n};\mb{Z})$ corresponding to the canonical orientation is $e_1f_1\dotsm e_\ell f_\ell$, so $\int_{T^{2n}} c_1(L_{2n})^n = n!$.
\end{proof}

The quaternionic analogs of Chern classes are the \emph{symplectic Pontryagin classes} $q_i$~\cite[\S 9.6]{BH58}, which generate the integral cohomology of $\B\Sp$:
\begin{equation}
\label{SpPC}
    H^*(\B\Sp;\mb{Z})\cong\mb{Z}[q_1,q_2,\ldots].
\end{equation}
Their degrees are given by $|q_i|=4i$. Not only do the classes $p_i$ and $q_i$ have the same degree, but they also coincide rationally in the following sense:

\begin{prop}
\label{Qcoincide}
Let $\B{f}\colon \B\Sp\to \B\O$ be the forgetful map and $(\B{f})^*\colon H^*(\B\O;\mb{Q})\to H^*(\B\Sp; \mb{Q})$ be the pullback map in cohomology. Then
\begin{enumerate}[(i)]
    \item\label{OSp_pull} $(\B{f})^*$ is a $\mb{Q}$-algebra isomorphism, and
    \item\label{ciQx} there are numbers $a_i\in\mb{Q}^\times$ such that $(\B{f})^*(p_i) = a_i q_i$.
\end{enumerate}
\end{prop}
\begin{proof}
Part~\eqref{ciQx} directly implies~\eqref{OSp_pull}, so we focus on~\eqref{ciQx}. Borel--Hirzebruch~\cite[\S 9.7]{BH58} show that, in rational cohomology, both $p_i$ and $q_i$ satisfy the same Whitney sum formula. That is, suppose $\{r_i\}$ is either of $\{p_i\}$ or $\{q_i\}$ and $E,F\to X$ are vector bundles (with $\Sp$-structures if $\{r_i\} = \{q_i\}$). Then
\begin{equation}
    r_n(E\oplus F) = \sum_{i+j = n} r_i(E)r_j(F)\in H^{4n}(X;\mb{Q}).
\end{equation}
Thus if we can show~\eqref{ciQx} for $i = 1$, it holds for all $i$ courtesy of the splitting principle for $\Sp$ (see~\cite[Example 9]{May05} or~\cite[Theorem 5.3.10]{BG10}). So, we want to show that $(\B{f})^*(p_1) = a_1q_1$ for some $a_1\in\mb{Q}^\times$. Since $H^4(\B\Sp;\mb{Q})\cong\mb{Q}$, generated by $q_1$, it suffices to show $(\B{f})^*(p_1)\ne 0$, i.e.\ to exhibit any $\Sp$-structured vector bundle with nonzero first Pontryagin class. An example such bundle is the tautological quaternionic line bundle over $\mathbb{HP}^n$ (see, e.g., ~\cite{Ale04} or~\cite[Remark 5.11]{FH21}).
\end{proof}

By considering the Chern--Dold character
\begin{equation}
    \cd_\KSp\colon \KSp\to\KSp_\mb{Q}\simeq H\mb{Q}[\![q_1]\!],
\end{equation}
we obtain a \emph{symplectic Pontryagin character} $\sy\in H^*(\B\Sp;\mb{Q})$ representing the homomorphism on cohomology induced by $\cd_\KSp$. The Chern and Pontryagin character formulas derive from the Whitney sum formula (as the associated Chern--Dold characters are homomorphisms out of $K$-theory). Since the Whitney sum formula holds rationally for both $p_i$ and $q_i$, we likewise deduce that
\begin{equation}
    \sy=\sum_{k\geq 1}\frac{s_k(q_1,\ldots,q_k)}{k!}.
\end{equation}
However, this is where we hit a snag. We know that \begin{align}
    (\B{f})^*(\po)=\sum_{k\geq 1}\frac{s_k(a_1q_1,\ldots,a_kq_k)}{k!},
\end{align}
but comparing this to $\sy$ would require an explicit calculation of the units $a_i$ in order to determine the relationship between $s_k(x_1,\ldots,x_k)$ and $s_k(a_1x_1,\ldots,a_kx_k)$.

We now compute another requisite characteristic class, namely the first Pontryagin class of the $\SU(2)$-bundle $Q_{4n}\to T^{4n}$. The first Pontryagin class of an $\SU(2)$-bundle is defined to be the first Pontryagin class of the associated quaternionic line bundle, i.e.\ for the defining representation of $\SU(2)$ on $\mb{H}$.
\begin{lem}
\label{pullback_p1}
Let $P\to X$ be a $\U(1)$-bundle. Then the first Pontryagin class of the induced $\SU(2)$-bundle $P\times_{\U(1)}\SU(2)$ is
\begin{equation}
    p_1(P\times_{\U(1)}\SU(2)) = -2 c_1(P)^2 \in H^4(X;\mb{Z}).
\end{equation}
\end{lem}

\begin{proof}
It suffices to prove this for the universal bundle $P = \mathrm E\U(1)\to \B\U(1)$. The class $p_1(P\times_{\U(1)}\SU(2))$ equals the pullback of $p_1$ from $\B\SU(2)$ to $\B\U(1)$, so we want to show this equals $-2c_1^2$, where $c_1\in H^2(\B\U(1);\mb{Z})$ denotes the first Chern class of the defining representation of $\U(1)$.

To compute the pullback map, first restrict $\mb{H}$ from an $\SU(2)$-representation to a $\U(1)$-representation along the map sending an element $z\in\U(1)$ to the matrix $\begin{psmallmatrix}z&0\\0&z^{-1}\end{psmallmatrix}$, which is how $z$ acts on $\mb{H}$ on the left in the basis $\{1, j\}$ for $\mb{H}$ as a $\mb{C}$-vector space, where scalar multiplication by $\mb{C}$ is on the right.
Then $\mb{H}$ pulls back to $\mathcal O(1)\oplus\mathcal O(-1)$: the associated bundle to the direct sum of the defining representation of $\U(1)$ and its conjugate.

By definition, $p_1(V) = c_2(V\otimes_{\mb{R}}\mb{C})$. Since $V\coloneqq \mathcal O(1)\oplus\mathcal O(-1)$ has a complex structure, \eqref{eq:complexify} implies that
\begin{equation}
\label{reduced_almost_final}
\begin{aligned}
    p_1(\mb{H}|_{\B\U(1)}) &= c_2((\mathcal O(1)\oplus\mathcal O(-1))\otimes_{\mb{R}}\mb{C})\\
    & = c_2(\mathcal O(1)\oplus\mathcal O(-1) \oplus \mathcal O(-1)\oplus\mathcal O(1)).
\end{aligned}
\end{equation}
By definition, $c_1(\mathcal O(1))$ is the standard generator $c_1$ of $H^2(\B\U(1);\mb{Z})$, and since $\overline{\mathcal O(1)}\cong \mathcal O(-1)$, we have $c_1(\mathcal O(-1)) = -c_1$. Apply the Whitney sum formula to~\eqref{reduced_almost_final} and plug in the values of $c_1(\mathcal O(\pm 1))$ to conclude $p_1(\mb{H}|_{\B\U(1)}) = -2c_1^2$.
\end{proof}
Combining \cref{T2y_charclass,pullback_p1}, we immediately deduce:
\begin{cor}
\label{pontryagin_cor}
    \[\int_{T^{4n}}p_1(Q_{4n})^n\neq 0.\]
\end{cor}

\section{$\Sp^x$ and $\Spin^x$ structures on hyperkähler manifolds}\label{sec:preliminaries}
In this section, we use the hyperkähler structure on $\Kt^{[n]}$ to define $\Sp^x$ and $\Spin^x$ structures on the manifolds in the generating sets $B_*^x$, for $x\in\{r,c,h\}$. As is standard in bordism theory, we need to fuss, at least to a minor degree, about whether we consider stable \emph{tangential} or stable \emph{normal} structures. 

Let $\displaystyle\O\coloneqq\colim_{n\to\infty}\O(n)$, the colimit of the orthogonal groups. Then $\B\O$ classifies virtual vector bundles up to stable equivalence (i.e.\ identifying $E$ and $E\oplus\underline{\mb{R}}$), so for every virtual vector bundle $E\to M$ one has an associated principal $\O$-bundle $P_E\to M$.
\begin{defn}
Let $\rho\colon G\to\O$ be a homomorphism of topological groups
\begin{enumerate}[(i)]
    \item A $G$-structure on a virtual vector bundle $E\to M$ is the data of a principal $G$-bundle $\widetilde P_E\to M$ and an isomorphism of principal $\O$-bundles
    \begin{equation}
        \widetilde P_E\times_G \O \overset\cong\longrightarrow P_E.
    \end{equation}
    \item A (stable) tangential $G$-structure on a manifold $M$ is a $G$-structure on $TM$.
    \item By work of Whitney~\cite{Whi44} and Wu~\cite{Wu58}, for $N\gg 0$ there is a unique isotopy class of embeddings $M\hookrightarrow S^N$; let $\nu$ denote the normal bundle to this inclusion, which is well-defined up to isomorphism.
    Composing with the equatorial inclusion $S^N\hookrightarrow S^{N+1}$ appends a trivial summand to $\nu$, so the stable equivalence class of $\nu$ is independent of $N$. A (stable) normal $G$-structure on $M$ is a $G$-structure on $\nu$.
\end{enumerate}
\end{defn}
As usual, there is a homotopy type of each of these kinds of $G$-structures on a given bundle or manifold, and we consider two $G$-structures equivalent if they lie in the same connected component.

Stable tangential structures are more natural in the geometer's lens, but
bordisms of stable tangential structures are modeled by Madsen--Tillmann spectra rather than Thom spectra.
In the course of showing that certain forgetful maps between our bordism spectra are highly structured, we found it more convenient to work with Thom spectra and thus with normal bordism, although similar structure should exist in the tangential case.
The results of this section imply that \cref{thm:main} is true for both types of bordism.

To define our normal $\Sp^x$ and $\Spin^x$ structures, we begin with some basic definitions and facts about hyperkähler structures.

\begin{defn}
    An \emph{almost quaternionic structure} on a manifold $M$ is the data of three almost complex structures $I,J,K$ such that $I^2=J^2=K^2=IJK=-1$. If in addition $M$ admits a Riemannian metric $g$ such that $I,J,K$ are each
    kähler with respect to $g$, then $(g, I, J, K)$ is called a \emph{hyperkähler structure} on $M$.
\end{defn}

Every almost quaternionic manifold is Spin$^h$, but such manifolds need not be Spin in dimensions $8n+4$ \cite[Lemma 3.9]{Bar99}. However, we can say more about hyperkähler manifolds.
\begin{prop}
\label{HK_to_Sp}
A hyperkähler manifold has a canonical normal $\Sp$-structure.
\end{prop}
\begin{proof}
    The holonomy of a hyperkähler manifold is contained in $\Sp(n)$ by \cite[p.~292]{Cal79}, which induces a tangential $\Sp(n)$-structure by \cite[Proposition 1.9]{Joy99}. Stabilizing gives us a tangential $\Sp$-structure, and the perp map $V\mapsto V^\perp$ sends this tangential $\Sp$-structure to a normal $\Sp$-structure.
\end{proof}

\begin{constr}
\label{product_normal_Sp}
For all $k$, the tangent bundle to the torus $T^k$ may be canonically trivialized using the Lie group framing, as $T^k\cong\U(1)^k$. Therefore, for any manifold $M$, there is an isomorphism of stable vector bundles
\begin{equation}
    \nu_{M\times T^k} \overset\sim\longrightarrow p^*\nu_M,
\end{equation}
where $p\colon M\times T^k\to M$ is projection onto the first factor. Thus if $M$ is hyperkähler, the normal $\Sp$-structure on $M$ defined in \cref{HK_to_Sp} and the Lie group framing on $T^k$ combine to define a normal $\Sp$-structure on $M\times T^k$.
\end{constr}

In dimension 4, the only compact hyperkähler manifolds (up to diffeomorphism) are $T^4$ and $\Kt$. In higher dimensions, Hilbert schemes of points on $T^4$ and $\Kt$ are again hyperkähler, as proved by Beauville.

\begin{thm}[{Beauville~\cite[Théorème 3]{Bea83}}]
\label{hilb_K3_HK}
For all $n$, $\Kt^{[n]}$ is hyperkähler.
\end{thm}

Using the hyperkähler structure on $\Kt^{[n]}$, we now define normal $\Sp^x$-structures on the elements of $B^x_*$.
\begin{defn}
\label{spx_structures}
For $x\in\{r,c,h\}$, we define normal $\Sp^x$-structures on the elements of $B_*^x$ as follows:
\begin{enumerate}[(i)]
    \item By \cref{hilb_K3_HK}, all elements of $B_*^r$ are hyperkähler; give them the $\Sp^r = \Sp$-structures from \cref{HK_to_Sp}.
    \item An $\Sp^c$-structure is an $\Sp$-structure and a principal $\U(1)$-bundle. All elements of $B_*^c$ are products of hyperkähler manifolds with $T^{2k}$ for some $k$; give them the $\Sp$-structures from \cref{product_normal_Sp} and the principal $\U(1)$-bundle induced from the line bundle~\eqref{torus_line}.
    \item The case $x = h$ is completely analogous to $x = c$, except that, instead of a $\U(1)$-bundle, we use the $\SU(2)$-bundle~\eqref{torus_SU2}.
\end{enumerate}
\end{defn}

We can then put a new spin on \cref{spx_structures}: the holonomy of a hyperkähler manifold of dimension $4n$ is a subgroup of $\Sp(n)$ \cite[p.~292]{Cal79}, so the inclusion $\Sp(n)\subseteq\Spin(4n)$\footnote{In fact, this is an instance of the forgetful map $\widetilde f_n$ that we will define in \cref{how_to_forget}.} provides a canonical Spin structure on every hyperkähler manifold (see e.g.~\cite[Proposition 1.9]{Joy99}). Combining this with the framing on $T^k$ like in \cref{product_normal_Sp} and passing through quotients, we have:
\begin{cor}
\label{spin_structures_on_B}
For $x\in\{r,c,h\}$, the normal $\Sp^x$-structures defined on the elements of $B_*^x$ in \cref{spx_structures} induce normal Spin\textsuperscript{$x$} structures on those manifolds.
\end{cor}
\begin{proof}
    For $x=r,c,h$, the canonical Spin structure on the hyperkähler manifold $\Kt^{[n]}$, together with the framing on $T^k$ and the bundles we placed on it in \cref{notn:bundles}, induces normal $\Spin$, $\Spin\times\U(1)$, and $\Spin\times\SU(2)$ structures, respectively, on the elements of $B^x_*$. We are done if $x=r$. If $x\in\{c,h\}$, then we pass through the quotients $\Spin(n)\times\U(1)\to\Spin^c(n)$ and $\Spin(n)\times\SU(2)\to\Spin^h(n)$ (see~\eqref{spin_quotients}) to finish.
\end{proof}
%

\section{Justifying $\Sp^x$}\label{sec:justification}
Readers with aesthetic sensitivities might bristle at the discrepancy between $\Sp^c=\Sp\times\U(1)$ and $\Spin^c=\Spin\cdot\U(1)$ (and likewise for $\Sp^h$ and $\Spin^h$). The reason for this difference is that there are no suitable connecting maps
\begin{equation}
\begin{aligned}
\Sp(n)\cdot\U(1)&\to\Sp(n+1)\cdot\U(1),\\
\Sp(n)\cdot\SU(2)&\to\Sp(n+1)\cdot\SU(2),
\end{aligned}
\end{equation}
as we will shortly explain. As a consequence, we cannot form a colimit to christen $\Sp\cdot\U(1)$ or $\Sp\cdot\SU(2)$, so there is no notion of an $\Sp\cdot\U(1)$ or $\Sp\cdot\SU(2)$ structure on a stable vector bundle. This forbids us from defining $\Sp\cdot\U(1)$-bordism and $\Sp\cdot\SU(2)$-bordism in the usual sense.

We will see in \cref{lem:classifying spaces} that the quotient map $G\times K\to G\cdot K$ is a rational homotopy equivalence. Therefore, when there \emph{is} a well-defined notion of $G\cdot K$-bordism, a generating set for rational $G\times K$-bordism yields a generating set for rational $G\cdot K$-bordism. This is our justification for defining $\Sp^c=\Sp\times\U(1)$ rather than $\Sp\cdot\U(1)$, and likewise for $\Sp^h$.

Assume for the purposes of contradiction that we did have connecting maps of the desired form for $\Sp(n)\cdot \U(1)$ structures. To define stable tangential structures, the maps connecting dimensions $4n$ and $4(n+1)$ should commute with the maps from the unstable structure groups to the appropriate orthogonal groups. However, a diagram of the form
\begin{equation}
    \begin{tikzcd}
        \Sp(n) \cdot \U(1) \ar[r] \ar[d] & \Sp(n+1)\cdot \U(1) \ar[d] \\
        \O(4n) \ar[r] & \O(4n+4)
    \end{tikzcd}
\end{equation}
cannot commute. Let us first look at the left and bottom arrows in the diagram.
In order to be compatible with the quotient inherent in the definition of $\Sp(n)\cdot \U(1)$, a pair $(M,z)$ of a $4n\times 4n$ symplectic matrix (i.e.~matrix preserving the standard hermitian form on $\mb{H}$) and a complex number $z\in U(1)$ must map to an orthogonal transformation that takes a vector $v\in \mathbb{H}^n$ to the vector $Mv\bar{z} \in \mathbb{H}^n$. Composing with the bottom map, we end up with the orthogonal transformation mapping $v\oplus y\in \mathbb{H}^n\oplus \mathbb{H}$ to the vector $M(v\bar{z}\oplus y) \in \mathbb{H}^{n+1}$. That is, $z$ must map to a transformation that preserves the $(n+1)\textsuperscript{st}$ quaternionic component.

Meanwhile, any pair $(N,z)$ in the top right must be sent to an orthogonal transformation sending a vector $w\in\mathbb{H}^{n+1}$ to the vector $Nw\bar{z}$. In particular, the complex number $z$ acts nontrivially on the $(n+1)$st component.

The argument for $\Sp(n)\cdot\SU(2)$ structures is similar. Intuitively, the reason these constructions fail is that $-1$ is not in the kernel of the map $\Sp\to\O$. Meanwhile, $\Spin\cdot U(1)$ and $\Spin\cdot\SU(2)$ give well-defined stable tangential structures in part because $-1$ is in the kernel of the map $\Spin\to \O$.

\begin{rem}
While it is not possible to form a stable notion of bordism for $\Sp(n)\cdot K$ (where $K=\U(1)$ or $\SU(2)$), the Madsen--Tillmann spectrum associated to the maps $\Sp(n)\cdot K\to\O(4n)$ \emph{does} provide an unstable notion of bordism. The $4n$th homotopy group of this spectrum classifies manifolds with (tangential) $\Sp(n)\cdot K$-structure up to bordism, where $X$ is bordant to $Y$ if there exists a $(4n+1)$-manifold $Z$ such that
\begin{enumerate}[(i)]
\item $\partial Z=X\amalg Y$ (with respect to the $\Sp(n)\cdot K$-structures on $X$, $Y$, and $Z$) and
\item there is a nowhere vanishing vector field on $Z$ that restricts to the inward and outward normal vector fields on $X$ and $Y$, respectively.
\end{enumerate}
To define an $\Sp(n)\cdot K$-structure on $Z$, we use the fact that the vector field splits a trivial summand off of $TZ$ and so reduces its structure group from $\O(4n+1)$ to $\O(4n)$, which we then want to lift across $\Sp(n)\cdot K\to\O(4n)$. See \cite{BS14,HSV25} for papers discussing this sort of bordism.
\end{rem}

\section{Isomorphisms of rational bordism}\label{sec:bordism isomorphisms}
In this section, we will prove that for $x\in\{r,c,h\}$, the forgetful map $\Omega^{\Sp^x}_*\otimes\mb{Q}\to\Omega^{\Spin^x}_*\otimes\mb{Q}$ is an isomorphism. We begin with a lemma that allows us to pass from rational bordism to the rational homology of the classifying space.

\begin{lem}\label{lem:bordism to cohomology}
Let $G$ be a topological group together with a map to $\SO$, so that bordism groups $\Omega_*^G$ of manifolds with $G$-structure (on the stable normal bundle) are defined. If $G$ is finite type, then there are natural isomorphisms $\Omega^G_n\otimes\mb{Q}\xrightarrow{\cong} H_n(\B{G};\mb{Q}) \xrightarrow{\cong} (H^n(\B{G};\mb{Q}))^\vee$ (as $\mb{Q}$-vector spaces) for all $n$, where $(\text{--})^\vee$ denotes the $\mb{Q}$-linear dual.
\end{lem}
\begin{proof}
Recall that given a stable tangential structure $G$, we have a natural isomorphism $\Omega^G_*\cong\pi_*(\M{G})$, where $\M$ denotes the Thom spectrum functor.
By the rational Hurewicz theorem, we have a natural isomorphism
\begin{equation}
\pi_*(\M{G})\otimes\mb{Q}\xrightarrow\cong H_*(\M{G};\mb{Q}).
\end{equation}
The Thom isomorphism then gives natural isomorphisms
\begin{equation}
H_*(\M{G};\mb{Q})\xrightarrow\cong H_*(\B{G};\mb{Q})
\end{equation}
(see e.g.~\cite{MR81}). Finally, we wish to establish a natural isomorphism $H_n(\B{G};\mb{Q})\cong (H^n(\B{G};\mb{Q}))^\vee$ for all $n$. We will use the universal coefficient theorem, which gives us natural short exact sequences
\begin{equation}
    0\to\Ext^1_\mb{Q}(H_{n-1}(\B{G};\mb{Q}),\mb{Q})\to H^n(\B{G};\mb{Q})\to\Hom_\mb{Q}(H_n(\B{G};\mb{Q}),\mb{Q})\to 0
\end{equation}
for all $n$. The $\Ext^1$ term vanishes, as higher $\Ext$ terms always vanish over a field, so there is a natural isomorphism $H^n(\B{G};\mb{Q}) \cong (H_n(\B{G};\mb{Q}))^\vee$. Since $G$ is finite type, $H_n(\B{G};\mb{Q})$ is a finite dimensional $\mb{Q}$-vector space, so the canonical map $H_n(\B{G};\mb{Q}) \to (H^n(\B{G};\mb{Q}))^\vee$ to the double dual is an isomorphism.
\end{proof}

Soon, we will apply \cref{lem:bordism to cohomology} to $G\in\{\Sp^x,\Spin^x\}$. Before we do so, we need to say a few words about the classifying spaces of these groups.

\begin{lem}\label{lem:classifying spaces}
Let $G,K$ be two path-connected topological groups with $\{\pm 1\}$ as a central subgroup. Then the map $\pi\colon \B{G}\times \B{K}\to \B(G\cdot K)$ induced by the quotient
\begin{equation}
    G\times K\longrightarrow (G\times K)/\{\pm 1\} = G\cdot K
\end{equation}
is a rational homotopy equivalence.
\end{lem}
In particular, $\pi$ is an isomorphism on rational homology and cohomology, and this isomorphism is natural in $G$ and in $K$.
One can prove this using the transfer map, but we give a spectral sequence argument.
\begin{proof}
To begin, recall that $\B(G\times K)\simeq\B{G}\times\B{K}$. 
The classifying space functor converts the short exact sequence
\begin{subequations}
\begin{equation}
    1\longrightarrow\{\pm 1\} \longrightarrow G\times K \longrightarrow G\cdot K\longrightarrow 1
\end{equation}
into a fiber sequence
\begin{equation}
\label{preSerre}
\B{\{\pm 1\}}\longrightarrow \B{G}\times \B{K}\longrightarrow \B(G\cdot K).
\end{equation}
\end{subequations}
Apply the $\mb{Q}$-coefficients homological Serre spectral sequence to~\eqref{preSerre}. The $E_2$-page is
\begin{equation}
\label{GKserre}
    E^2_{*,*} = H_*\bigl(\B(G\cdot K); H_*(\B{\{\pm 1\}}; \mb{Q})\bigr).
\end{equation}
For any finite group $H$, Maschke's theorem implies that $H^*(\B{H}; \mb{Q})$ vanishes in positive degrees; the universal coefficient theorem implies that the same is true for the positive-degree $\mb{Q}$-homology of $\B{H}$. Thus~\eqref{GKserre} collapses, implying that the edge homomorphism $H_*(\B{G}\times\B{K}; \mb{Q})\to H_*(\B(G\cdot K);\mb{Q})$, which is exactly the map $\pi$ induces on homology, is an isomorphism. Since $G$ and $K$ are connected, so are $G\times K$ and $G\cdot K$, implying $\B{G}\times \B{K}$ and $\B(G\cdot K)$ are simply connected, so this rational homology isomorphism upgrades to a rational homotopy equivalence.
\end{proof}

After some preparation, we will prove in \cref{prop:rational equivalence} that the maps of spectra $\M\Sp^x\to\M\Spin^x$ are rational homotopy equivalences. We will then use this to prove that the forgetful map $\Sp\to\O$ induces isomorphisms $\Omega^{\Sp^x}_*\otimes\mb{Q}\to\Omega^{\Spin^x}_*\otimes\mb{Q}$.
\begin{constr}
\label{how_to_forget}
Because $\Sp$ is connected, the forgetful map\footnote{The composition $\Sp\to\U\to\O$ given by the choice of embedding $\mb{C}\to\mb{H}$ from \S\ref{conventions}, followed by the usual forgetful map from complex-valued to real-valued matrices, equals $f\colon \Sp\to\O$.} $f\colon \Sp\to\O$ lands in the connected component of the identity in $\O$, which is $\SO$. Since $\Sp$ is simply connected, the map $\Sp\to\SO$ lifts uniquely as a group homomorphism to the universal cover of $\SO$, which is $\Spin$. That is, we have a group homomorphism $\widetilde f\colon\Sp\to\Spin$. In exactly the same way we obtain forgetful maps $\widetilde f_n\colon\Sp(n)\to\Spin(n)$.
\end{constr}

\begin{rem}
We are going to use the forgetful map $\tilde{f}:\Sp\to\Spin$ to construct rational homotopy equivalences $\tilde{F}^x:\M\Sp^x\to\M\Spin^x$ (\cref{prop:rational equivalence}) and subsequently ring (or module) isomorphisms $\Omega^{\Sp^x}_*\otimes\mb{Q}\to\Omega^{\Spin^x}_*\otimes\mb{Q}$ (\cref{cor:bordism isomorphism}). As we mentioned in Footnote~\ref{attrfoot}, these ideas are not original to us, as the forgetful map from $\Sp$ bordism to Spin bordism goes back at least to Conner--Floyd~\cite{CF66} and Stong~\cite{Sto68}, and these rational homotopy equivalences are known. However, the details have never been written down as far as we can tell, and these details are necessary for the proof of \cref{cor:bordism isomorphism}.
\end{rem}

\begin{lem}
\label{spaces_Qhom}
The map $\B{\widetilde f}\colon \B{\Sp}\to \B{\Spin}$ induced by forgetting from an $\Sp$-structure to a Spin structure is a rational homotopy equivalence.
\end{lem}
\begin{proof}
Because $\B\Spin$ and $\B\Sp$ are simply connected, it suffices to show that the pullback map $\B{\widetilde f}$ induces on rational cohomology is an isomorphism.

Recall from~\eqref{SpPC} that $H^*(\B\Sp;\mb{Z})$ is a polynomial ring on the symplectic Pontryagin classes $q_i$, so the universal coefficient theorem implies that $H^*(\B\Sp;\mb{Q})\cong\mb{Q}[q_1,q_2,\dotsc]$ with $|q_i| = 4i$. The map $\Spin(n)\to\SO(n)$ induces an isomorphism $H^*(\B\Spin;\mb{Q})\cong H^*(\B\SO;\mb{Q})$~\cite[\S 1]{Tho62}. By e.g.~\cite[Theorem 15.9]{MS74}, we then have $H^*(\B\Spin;\mb{Q})\cong\mb{Q}[p_1,p_2,\dotsc]$, where $p_i$ is the $i^{\mathrm{th}}$ Pontryagin class with $|p_i| = 4i$. Thus we are done if $(\B{\widetilde f})^*(p_i) = a_iq_i$ (with $a_i\in\mb{Q}^\times$) for all $i$, which is \cref{Qcoincide}.
\end{proof}
Recall the forgetful maps $\widetilde f^c$ and $\widetilde f^h$ from~\eqref{CH_forget}.
\begin{cor}
\label{spaces_x_Qhom}
The forgetful maps $\B{\widetilde f^x}\colon \B\Sp^x\to \B\Spin^x$, for $x\in\{r,c,h\}$, are rational homotopy equivalences.
\end{cor}
\begin{proof}
By the Künneth formula and \cref{spaces_Qhom}, the maps $(\B{\widetilde f},\mathrm{id}_X)\colon \B\Sp\times X\to \B\Spin\times X$, for $X = \B\U(1)$ or $\B\SU(2)$, are rational homology equivalences. Then compose with the isomorphism in \cref{lem:classifying spaces} to conclude that $\B{\widetilde f^x}$ is also a rational homology equivalence. Since the domain and codomain of $\B{\widetilde f^x}$ are both simply connected for $x\in\{r,c,h\}$, this lifts to imply $\B{\widetilde f^x}$ is a rational homotopy equivalence.
\end{proof}

\begin{prop}\label{prop:rational equivalence}
    For $x\in\{r,c,h\}$, the forgetful map $\widetilde f^x\colon \Sp^x\to\Spin^x$ induces a rational homotopy equivalence
    \begin{equation}\label{eq:rational equivalence}
    \widetilde F^x\colon \M\Sp^x\to\M\Spin^x.
    \end{equation}
\end{prop}
\begin{proof}
Let $H$ be a topological group with a map $\rho\colon H\to\SO$. The Thom isomorphism $H_*(\B{H};\mb{Q})\xrightarrow{\cong} H_*(\M{H};\mb{Q})$ is natural in $(H, \rho)$,\footnote{See, for example,~\cite[Proof of Lemma 6.3]{Mil65}: naturality of the Thom isomorphism follows from naturality of the Thom class.} so the rational homology equivalences implied by \cref{spaces_x_Qhom} imply that the respective forgetful maps induce isomorphisms
\begin{equation}
    H_*(\M\Sp^x;\mb{Q}) \xrightarrow{(\widetilde F^x)_*}
    H_*(\M\Spin^x;\mb{Q})
\end{equation}
of graded $\mb{Q}$-vector spaces. The rational Whitehead theorem then finishes the proof.
\end{proof}

\begin{defn}\label{defn:plus-c}
Let $G(n)$ be either of $\Sp(n)$ or $\Spin(n)$. Define the map
\begin{equation}\label{spin_oplus}
    \begin{aligned}
        \oplus^c\colon G(m)\times\U(1) \times G(n)\times\U(1) &\longrightarrow G(m+n)\times\U(1)\\
        (A_1, z_1, A_2, z_2) &\longmapsto (A_1\oplus A_2, z_1z_2).
    \end{aligned}
\end{equation}
When $G(n) = \Spin(n)$, the reader can check that if $\lambda,\mu\in\mb{Z}/2$, the elements
\begin{subequations}
\begin{equation}
    ((-1)^\lambda, (-1)^\lambda, (-1)^\mu, (-1)^\mu)
\end{equation}
are in the kernel of the composition
\begin{equation}
    \Spin(m)\times\U(1)\times\Spin(n)\times\U(1) \xrightarrow{\oplus^c} \Spin(m+n)\times\U(1) \xrightarrow{q} \Spin^c(m+n),
\end{equation}
so $\oplus^c$ descends through the quotient by these elements to define a map
\begin{equation}\label{spinc_oplus}
    \oplus^c\colon\Spin^c(m)\times\Spin^c(n)\longrightarrow \Spin^c(m+n).
\end{equation}
\end{subequations}
\end{defn}
The map $\oplus^c$ induces the standard Spin$^c$ structure on the direct sum of Spin$^c$ vector bundles $V$ and $W$, whose determinant line bundle is the tensor product of the determinant line bundles of $V$ and $W$.

Next, we will prove that the maps $\widetilde F^x\colon \M\Sp^x\to\M\Spin^x$ are compatible with the relevant ring or module structures in \cref{prop:e-infty}.

\begin{lem}
\label{lie_gps_rings}
Let $\widetilde f_n\colon\Sp(n)\to\Spin(4n)$ be the forgetful map from \cref{how_to_forget}.
If $(x,y)\in\{(r,r),(r,c),(r,h),(c,c)\}$, the following diagram of Lie groups commutes:
\begin{equation}
\begin{tikzcd}\label{sp_spin_plus_eq}
	{\Sp^x(m)\times\Sp^y(n)} & {\Sp^y(m+n)} \\
	{\Spin^x(4m)\times\Spin^y(4n)} & {\Spin^y(4(m+n)).}
	\arrow["\oplus^*", from=1-1, to=1-2]
	\arrow["{(\widetilde f_m, \widetilde f_n)}"', from=1-1, to=2-1]
	\arrow["{\widetilde f_{m+n}}", from=1-2, to=2-2]
	\arrow["\oplus^*", from=2-1, to=2-2]
\end{tikzcd}\end{equation}
Here, $\oplus^*=\oplus$ unless $x=y=c$, in which case $\oplus^*=\oplus^c$ (as defined in~\eqref{spinc_oplus}).
\end{lem}
\begin{proof}
%
Differentiate~\eqref{sp_spin_plus_eq} to a commutative diagram of Lie algebras, where it asks, does $\mathrm d_ef_n\colon\mathfrak{sp}(n)\to\mathfrak o(4n)$ commute with direct sums? This map is the inclusion of the quaternionic skew-hermitian $4n\times 4n$ matrices into all $4n\times 4n$ matrices, so yes, it does commute with direct sums. Differentiation defines an equivalence of categories of connected, simply connected Lie groups and finite-dimensional Lie algebras, which means we can lift to connected, simply connected Lie groups. This settles the $(x,y)=(r,r)$ case.

The $(r, c)$ and $(r, h)$ cases are essentially the same: it suffices to verify that
\begin{equation}
    (\mathrm d_e f_n, \mathrm{id})\colon
        \mathfrak{sp}(n)\oplus \mathfrak{g} \longrightarrow
        \mathfrak{o}(n)\oplus \mathfrak{g}
\end{equation}
commutes with direct sums in the first component, where $\mathfrak g\in\{\mathfrak u(1), \mathfrak{su}(2)\}$. There is slightly more to say because $\Sp^c(n)$, $\Spin^c(n)$, and $\Spin^h(n)$ are not simply connected, so in all cases we exponentiate to the universal covers, obtaining maps $\Sp\times\mb{R}\to\Spin\times\mb{R}$, resp.\ $\Sp^h\to\Spin\times\SU(2)$ commuting with direct sums in the first term. The reader can then directly check that these descend to the groups we are actually interested in.

This leaves only the $(c, c)$ case.
Differentiating the group operation on $\U(1)$ gives addition on $\mathfrak u(1)\cong i\mb{R}$, so  on Lie algebras we have the map
\begin{equation}
    (\mathrm d_e f_n, +)\colon
        \mathfrak{sp}(n)\oplus \mathfrak{u}(1) \longrightarrow
        \mathfrak{o}(n)\oplus \mathfrak{u}(1),
\end{equation}
which certainly commutes with direct sums in the first variable. Then proceed as before, lifting to the universal cover and quotienting down to $\Sp^c$ and $\Spin^c$.
\end{proof}
Let $\hat\imath_n\colon\Sp(n)\hookrightarrow\Sp(n)\times\U(1) = \Sp^c(n)$ be the inclusion of the first factor. Similarly, let $i_n\colon\Spin(n)\to\Spin^c(n)$ be the inclusion of the first factor in $\Spin(n)\hookrightarrow\Spin(n)\times\U(1)$, followed by the quotient by $(-1, -1)$ to obtain $\Spin^c(n)$.

\begin{prop}
\label{THE_BIG_CUBE}
The following diagram commutes:
\[
\begin{tikzcd}[row sep=1em, column sep=1em, >=stealth]
	{\Sp(m)\times\Sp(n)} && {\Sp(m+n)} \\
	& {\Sp^c(m)\times\Sp^c(n)} && {\Sp^c(m+n)} \\
	{\Spin(4m)\times\Spin(4n)} && {\Spin(4(m+n))} \\
	& {\Spin^c(4m)\times\Spin^c(4n)} && {\Spin^c(4(m+n)).}
	\arrow["\oplus", from=1-1, to=1-3]
	\arrow["{(\hat\imath_{m}, \hat\imath_{n})}"', from=1-1, to=2-2]
	\arrow["{(\tilde f_m, \tilde f_n)}"', from=1-1, to=3-1]
	\arrow["{\hat\imath_{m+n}}", from=1-3, to=2-4]
	\arrow["{\widetilde f_{m+n}}"{pos=0.7}, from=1-3, to=3-3]
	\arrow["{\oplus^c}"{pos=0.3}, from=2-2, to=2-4, crossing over]
	\arrow["\oplus"{pos=0.3}, from=3-1, to=3-3]
	\arrow["{(\tilde f_{4m}^c, \tilde f_{4n}^c)}"{pos=0.7}, from=2-2, to=4-2, crossing over]
	\arrow["{\tilde f_{m+n}^c}", from=2-4, to=4-4]
	\arrow["{(i_{4m}, i_{4n})}"', from=3-1, to=4-2]
	\arrow["{i_{4(m+n)}}", from=3-3, to=4-4]
	\arrow["{\oplus^c}", from=4-2, to=4-4]
\end{tikzcd}\]
\end{prop}
\begin{proof}
The idea is to split this diagram into two subcubes. Factor the top face of the cube as
\begin{equation}\label{partcube1}
    \begin{tikzcd}
        \Sp(m)\times\Sp(n)\arrow[r,"{(\hat\imath_m,\hat\imath_n)}"]\arrow[d,"{\oplus}"] &\Sp^c(m)\times\Sp^c(n)\arrow[r,"{(\id,\id)}"]\arrow[d,"{\oplus^c\,\eqref{spin_oplus}}"]&\Sp^c(m)\times\Sp^c(n)\arrow[d,"{\oplus^c\,\eqref{spin_oplus}}"]\\
    \Sp(m+n)\arrow[r,"{\hat\imath_{m+n}}"]&\Sp^c(m+n)\arrow[r,"\id"]&\Sp^c(m+n).
    \end{tikzcd}
\end{equation}
Similarly, factor the bottom face of the cube as
\begin{equation}\label{partcube2}\begin{tikzcd}[column sep=0.4cm]
	{\Spin(4m)\times\Spin(4n)} & {\Spin(4m)\times\U(1)\times\Spin(4n)\times\U(1)} & {\Spin^c(4m)\times\Spin^c(4n)} \\
	{\Spin(4(m+n))} & {\Spin(4(m+n))\times\U(1)} & {\Spin^c(4(m+n)).}
	\arrow[hook, from=1-1, to=1-2]
	\arrow["\oplus"', from=1-1, to=2-1]
	\arrow["{(q, q)}", from=1-2, to=1-3]
	\arrow["{\oplus^c\,\eqref{spin_oplus}}", from=1-2, to=2-2]
	\arrow["{\oplus^c\,\eqref{spinc_oplus}}", from=1-3, to=2-3]
	\arrow[hook, from=2-1, to=2-2]
	\arrow["q", from=2-2, to=2-3]
\end{tikzcd}
\end{equation}
Both~\eqref{partcube1} and~\eqref{partcube2} commute directly by the definitions of $\oplus$, $\hat\imath_n$, $\oplus^c$, and $q$. This takes care of the top and bottom faces of the cube; the remaining four faces commute either by definition or by \cref{lie_gps_rings}.
\end{proof}
\begin{prop}\label{prop:e-infty}\hfill
\begin{enumerate}[(i)]
    \item The map $\widetilde F^r\colon\M\Sp\to\M\Spin$ from~\eqref{eq:rational equivalence} is an $\mb{E}_\infty$-ring map.
    \item The map $\widetilde F^c\colon \M\Sp^c\to\M\Spin^c$ from~\eqref{eq:rational equivalence} is an $\mb{E}_\infty$-algebra map over the $\mb{E}_\infty$-ring map $\M\Sp\to\M\Spin$.
    \item The map $\widetilde F^h\colon \M\Sp^h\to\M\Spin^h$ from~\eqref{eq:rational equivalence} is a module map over the $\mb{E}_\infty$-ring map $\M\Sp\to\M\Spin$.
\end{enumerate}
\end{prop}
\begin{proof}
Schwede~\cite[\S 6.1]{Sch18} constructs the Thom spectra $\mathrm M\O$, $\mathrm M\SO$, and $\mathrm M\U$ as commutative monoids in the symmetric monoidal category of orthogonal spectra.\footnote{Schwede constructs a more complicated object called a \emph{global orthogonal spectrum}, but this induces the structure of an orthogonal spectrum, and therefore we can ignore the additional global structure: see~\cite[\S 4.1]{Sch18}. See also Bohmann~\cite[Example 1.8]{Boh14} for another construction of $\mathrm{MU}$ in this manner.}\textsuperscript{,}\footnote{Mandell--May--Schwede--Shipley~\cite[Theorem 10.4]{MMSS01} construct a functor from the category of orthogonal spectra to the usual category of spectra, and likewise for commutative monoids in orthogonal spectra and $\mb{E}_\infty$-ring spectra (\textit{ibid.}, Theorem 0.5), and for modules over a commutative monoid $R$ and modules over the corresponding $\mb{E}_\infty$-ring spectrum (\textit{ibid.}, Corollary 0.6). Thus making these constructions in the world of orthogonal spectra suffices for our purposes.}\textsuperscript{,}\footnote{The first construction of Thom spectra as orthogonal spectra is due to May--Quinn--Ray~\cite[\S IV.2]{MQRT77}. We use Schwede's more recent construction so that we can more easily apply \cref{lie_gps_rings,THE_BIG_CUBE}.}
It is straightforward to generalize his construction as follows: given a natural number $k$ and a sequence of groups $\rho_n\colon G(n)\to\O(kn)$ for all $n$ with maps $i_n\colon G(n)\to G(n+1)$ such that the diagram
\begin{equation}
\begin{tikzcd}
	{G(n)} & {G(n+1)} \\
	{\O(kn)} & {\O(k(n+1))}
	\arrow["i_n", from=1-1, to=1-2]
	\arrow["\rho_n", from=1-1, to=2-1]
	\arrow["\rho_{n+1}", from=1-2, to=2-2]
	\arrow[from=2-1, to=2-2]
\end{tikzcd}
\end{equation}
commutes, one obtains an orthogonal spectrum model for the Thom spectrum $\mathrm MG$, and this is natural in the data $(G(n), i_n, \rho_n)$. Moreover:
\begin{itemize}
    \item Suppose that we have strictly associative direct-sum maps $\oplus\colon G(n)\times G(m)\to G(n+m)$ commuting with $\rho_\bullet$ and the direct-sum maps on $\O(kn)$, and suppose as well that we have maps $G(m+n)\to G(n+m)$ covering the maps $\O(k(m+n))\to \O(k(n+m))$, etc., witnessing the symmetric monoidality of the direct sum. Then
    $\mathrm MG$ has the structure of a commutative monoid object in orthogonal spectra, and this is natural in $(G(n), i_n, \rho_n, \oplus)$.
    \item Given $(G(n), i_n, \rho_n, \oplus)$ as above, another collection of data $(H(n), i_n', \rho_n')$ as above, and direct-sum maps $\oplus^{G,H}\colon G(n)\times H(m)\to H(n+m)$ commuting with the maps $\rho_\bullet$, $\rho_\bullet'$, and the direct-sum maps on $\O(kn)$, $\mathrm MH$ has the structure of an $\mathrm MG$-module spectrum, and this is natural in the data $(G(n), i_n, \rho_n, \oplus)$, $(H(n), i_n', \rho_n')$, and $\oplus^{G,H}$.
\end{itemize}
Thus the proposition follows from \cref{lie_gps_rings,THE_BIG_CUBE}, which provide the necessary maps and commutative diagrams.
\end{proof}
\begin{cor}\label{cor:bordism isomorphism}
    For $x\in\{r,c,h\}$, the forgetful map $\widetilde F^x$ induces an isomorphism
    \begin{equation}
    \Omega^{\Sp^x}_*\otimes\mb{Q}\xrightarrow{\cong}\Omega^{\Spin^x}_*\otimes\mb{Q}
    \end{equation}
    of rings (if $x\in\{r,c\}$) or of $\Omega^{\Sp}_*\otimes\mb{Q}$-modules (if $x=h$).
\end{cor}
\begin{proof}
    A rational homotopy equivalence (as given in \cref{prop:rational equivalence}) induces an isomorphism on rational homotopy groups. The isomorphism of ring or module structures then follows from \cref{prop:e-infty}.
\end{proof}

Using \cref{lem:bordism to cohomology,lem:classifying spaces,cor:bordism isomorphism}, we obtain a quick proof of the ranks of these bordism groups (c.f.~\cite{ABP67,BM23}).

\begin{cor}\label{cor:ranks}
    Let $G\in\{\Sp,\Spin\}$. Let $\mc{P}(n)$ denote the set of partitions of $n$, and let $P(n)\coloneqq|\mc{P}(n)|$. Then we have
    \begin{subequations}
    \begin{align}
        \rank\Omega^G_{4n}&= P(n),\\
        \rank\Omega^{G^c}_{4n}&=\rank\Omega^{G^c}_{4n+2}=\rank\Omega^{G^h}_{4n}=\sum_{m=0}^n P(m).
    \end{align}
    \end{subequations}
    Moreover, all degrees not listed above are of rank 0.
\end{cor}
\begin{proof}
    By \cref{cor:bordism isomorphism}, it suffices to treat the case of $G=\Sp$. By \cref{lem:bordism to cohomology}, we need to compute the rank of $H^*(\B{G^x};\mb{Q})$ in each degree. Due to the isomorphism
    \begin{equation}
        H^*(\B\Sp;\mb{Q})\cong\mb{Q}[p_1,p_2,\ldots]
    \end{equation}
    (see \cref{Qcoincide}),
    the real case amounts to counting the number of monomials (in $p_1,p_2,\ldots$) of degree $4n$, which is exactly $P(n)$. 
    
    For the complex and quaternionic cases, we first apply \cref{lem:classifying spaces} and the Künneth formula to write
    \begin{equation}\label{eq:kunneth}
    H_k(\B\Sp^x;\mb{Q})\cong\bigoplus_{i+j=k}H_i(\B\Sp;\mb{Q})\otimes H_j(\B{K^x};\mb{Q}),
    \end{equation}
    where $K^c=\U(1)$ and $K^h=\SU(2)$. Now by the isomorphisms $H_n(\B{A};\mb{Q})\cong H^n(\B{A};\mb{Q})$ for $A$ a topological group satisfying the conditions of \cref{lem:bordism to cohomology}, it follows that the rank of $H^k(\B{G^x};\mb{Q})$ is equal to the rank of
    \begin{equation}
        \bigoplus_{i+j=k} H^i(\B\Sp;\mb{Q})\otimes H^j(\B{K^x};\mb{Q}).
    \end{equation}
    This rank is in turn equal to the number of monomials of degree $k$ in
    \begin{align*}
        \mb{Q}[p_1,p_2,\ldots]&\otimes\mb{Q}[c_1],\tag{$x=c$}\\
        \mb{Q}[p_1,p_2,\ldots]&\otimes\mb{Q}[p_1'].\tag{$x=h$}
    \end{align*}
    For $k=4n$ or $4n+2$, computing these ranks thus amounts to counting the number of monomials in $p_1,p_2,\ldots$ of degree at most $4n$, and then shifting up to degree $4n$ or $4n+2$ by multiplying by powers of $c_1$ and $p_1'$.
\end{proof}

\section{Linear independence}\label{sec:bases}
We proved the first statement of \cref{thm:main} in \cref{cor:bordism isomorphism}. The next part of this theorem that we will prove is that $B^r_{4n}$ is a basis for $\Omega^{\Sp}_{4n}\otimes\mb{Q}$. For this, we need an auxiliary definition. We begin by recalling Milnor's computation of the rational complex bordism ring.

Let $\c_i\in H_{2i}(\B{\U}; \mb{Z})$ denote the dual of the $i\textsuperscript{th}$ Chern class $c_i$, so that $H_*(\B{\U};\mb{Q})\cong \mb{Q}[\c_1,\c_2,\dotsc]$.
\begin{thm}[{Milnor~\cite{Mil60}}]
\label{rational_U}
There is a ring isomorphism
\begin{equation}\label{rational MU}
    \Omega_*^\U \otimes\mb{Q} \xrightarrow{\cong} \mb{Q}[\c_1, \c_2, \dotsc]
\end{equation}
sending the rational bordism class $[M]$ of a stably almost complex manifold $M$ to the image of the rational fundamental class of $M$ in $H_*(\B{\U};\mb{Q})$ under the classifying map $M\to \B{\U}$.
\end{thm}

Milnor actually obtained an isomorphism to a polynomial ring over $\mb{Z}$, though we will not use this stronger statement.

\begin{defn}\label{even_subring}
Under the isomorphism \eqref{rational MU}, let
    \[\Omega^{\U,\even}_*\otimes\mb{Q}\cong\mb{Q}[\check c_2,\check c_4,\dots]\]
    denote the subring of $\Omega^{\U}_*\otimes\mb{Q}$ consisting of (rational complex bordism classes of) manifolds whose odd Chern classes are all trivial.
\end{defn} 

%
Our proof of \cref{thm:main} will begin by forgetting $\Sp$-structure for a moment and working in the rational complex bordism ring. This will allow us to use \cite{OSV22}, which will imply that the elements of $B^r_{4n}$ span $\Omega^\Sp_{4n}\otimes\mb{Q}$. Linear independence will then follow from the fact that $|B^r_{4n}|=\dim\Omega^\Sp_{4n}\otimes\mb{Q}$.

\begin{lem}\label{lem:Br is basis}
    The set $B^r_{4n}=\{\prod_{i=1}^a\Kt^{[n_i]}:(n_1,\ldots,n_a)\in\mc{P}(n)\}$ is a basis of $\Omega^{\Sp}_{4n}\otimes\mb{Q}$.
\end{lem}
\begin{proof}
    To begin, we will consider $B^r_{4n}$ as a set of elements in $\Omega^{\U,\even}_{4n}\otimes\mb{Q}$. Consider the forgetful map\footnote{This map is well-defined, as all odd Chern classes of $\Sp$ manifolds vanish.}
    \begin{equation}
        f\colon\Omega^{\Sp}_{4n}\otimes\mb{Q}\to\Omega^{\U,\even}_{4n}\otimes\mb{Q}
    \end{equation}
    given by forgetting the $\Sp$-structure of a manifold and only remembering the underlying complex structure determined by our choice of inclusion $\mb{C}\hookrightarrow\mb{H}$ (see~\S\ref{conventions}). We claim that $f$ is injective. Indeed, Equation~\eqref{eq:chern to pontryagin} implies a surjection in cohomology $H^*(\M\U; \mb{Q}) \to H^*(\M\Sp ; \mb{Q})$, so the dual map $H_*(\M\Sp; \mb{Q}) \to H_*(\M\U; \mb{Q})$ (which is precisely the map $f$) is injective.

    Next, note that $\Omega_{4n}^{\U,\even} \otimes \mb{Q}$ has dimension $|\mc{P}(n)|$, since as a ring we have
    \begin{equation}
        \Omega_*^{\U,\even}\otimes\mb{Q} \cong \mb{Q}[\c_2, \c_4, \ldots].
    \end{equation}
    Thus $\Omega^{\U,\even}_{4n}\otimes\mb{Q}$ and $\Omega^\Sp_{4n}\otimes\mb{Q}$ have the same rank by \cref{cor:ranks}. Since $f$ is an injection, it follows that $f$ is an isomorphism of $\mb{Q}$-vector spaces. The claim now follows from \cite[Theorem 1.1(a)]{OSV22}, which states that $f(B^r_{4n})$ generates $\Omega^{\U,\even}_{4n}\otimes\mb{Q}$, and the fact that $|B^r_{4n}|=\dim\Omega^{\U,\even}_{4n}\otimes\mb{Q}$.
\end{proof}

We have already seen in \cref{cor:ranks} that $|B^x_{4n}|=\dim\Omega^{\Sp^x}_{4n}\otimes\mb{Q}$ (for $x\in\{c,h\}$) and $|B^c_{4n+2}|=\dim\Omega^{\Sp^c}_{4n+2}\otimes\mb{Q}$. To complete our proof of \cref{thm:main}, it suffices to show that the elements of each of these sets are linearly independent. We will do this in three steps.

\begin{lem}[Step 1]\label{lem:step 1}
    For $x\in\{c,h\}$, let $\iota^x\colon\Omega^{\Sp}_*\otimes\mb{Q}\to\Omega^{\Sp^x}_*\otimes\mb{Q}$ be the map induced by the inclusion $\Sp\hookrightarrow\Sp^x$. Then $\iota^x$ is injective.
\end{lem}
\begin{proof}
    The inclusion $\Sp\hookrightarrow\Sp^x$ induces an injective group homomorphism
    \begin{equation}
        H_k(\B\Sp;\mb{Q})\to H_k(\B\Sp^x;\mb{Q})
    \end{equation}
    for each $k$ (e.g.~by Equation \eqref{eq:kunneth}). The injectivity of $\iota^x$ now follows from \cref{lem:bordism to cohomology}.
\end{proof}

\begin{notn}
\label{part_n}
For a partition $\lambda = (m_1,\dotsc,m_a) \in\mathcal P(n)$, let $p_\lambda\in H^{4n}(\B\O;\mb{Q})$ be
\begin{equation}
    p_\lambda\coloneqq p_{m_1}p_{m_2}\dotsm p_{m_a}.
\end{equation}
Under the identifications
\begin{equation}
    (\Omega_{4n}^\Sp\otimes\mb{Q})^\vee\cong (H_{4n}(\B\Sp;\mb{Q}))^\vee\cong H^{4n}(B\Sp;\mb{Q}),
\end{equation}
the class $p_\lambda$ corresponds to the functional $\varphi_\lambda\colon \Omega_{4n}^\Sp\otimes\mb{Q}\to\mb{Q}$ sending $M\mapsto \int_M p_\lambda(M)$. Now let
\begin{equation}
    \Phi_n\colon \Omega_{4n}^\Sp\otimes\mb{Q} \xrightarrow\cong \mb{Q}\cdot\mathcal P(n)
\end{equation}
be the function sending $M$ to the vector whose entry at $\lambda\in\mathcal P(n)$ is $\varphi_\lambda(M)$, which by \cref{cor:ranks} is an isomorphism.
\end{notn}


\begin{lem}[Step 2]\label{lem:step 2}
    Consider the functions
    \begin{subequations}
    \begin{equation}
    \begin{aligned}
    \tau^c_y\colon\Omega^{\Sp}_{4n}\otimes\mb{Q}&\to\Omega^{\Sp^c}_{4n+2y}\otimes\mb{Q}\\
    [M]&\mapsto[M\times T^{2y}_c].
    \end{aligned}
    \end{equation}
    and
    \begin{equation}
    \begin{aligned}
    \tau^h_y\colon\Omega^{\Sp}_{4n}\otimes\mb{Q}&\to\Omega^{\Sp^h}_{4n+4y}\otimes\mb{Q}\\
    [M]&\mapsto[M\times T^{4y}_h].
    \end{aligned}
    \end{equation}
    \end{subequations}
    Then $\tau^c_y$ and $\tau^h_y$ are injections for all $n,y\geq 0$.
\end{lem}
\begin{proof}
We will show that for $\tau = \tau_y^c,\tau^h_y$, the isomorphism $\Phi_n$ from \cref{part_n} factors as the composition of $\tau$ and another map, which implies $\tau$ is injective.

For an $\Sp^c$ manifold $M$, let $L_M$ denote the determinant line bundle, i.e.\ the complex line bundle induced from the $\U(1)$-bundle that is part of the $\Sp^c$ structure.
Let $\Psi_n\colon \Omega_{4n+2y}^{\Sp^c}\otimes\mb{Q}\to\mb{Q}\cdot\mathcal P(n)$ be the function sending an $\Sp^c$ manifold $M$ to the vector whose entry at $\lambda\in\mathcal P(n)$ is
\begin{equation}
    \int_M c_1(L_M)^y p_\lambda(M).
\end{equation}
We claim that $\Psi_n\circ \tau_y^c = (y!) \Phi_n$. As discussed above, proving the claim will finish the proof for $\Sp^c$. Unwinding the definitions of $\Phi_n$ and $\Psi_n$, for all partitions $\lambda$ of $n$ we want to calculate
\begin{equation}
\label{spc_integral}
    \int_{M\times T^{2y}} c_1(L_{M\times T^{2y}_c})^y p_\lambda(M\times T^{2y}_c)
\end{equation}
and show that it equals $(y!)\int_M p_\lambda(M)$. The pieces in the integrand of~\eqref{spc_integral} simplify:
\begin{itemize}
    \item Since the tangent bundle of $T^{2y}$ is trivializable and the Pontryagin classes are stable, $p_\lambda(M\times T^{2y}_c)$ equals the pullback of $p_\lambda(M)$ along the projection $M\times T^{2y}_c\to M$.
    \item By construction, in the $\Sp^c$ structure we placed on $M\times T^{2y}_c$, the bundle $L_{M\times T_c^{2y}}$ is the pullback of $L_{2n}$ along $M\times T^{2y}_c\to M$; therefore the Chern class also pulls back from $T^{2y}_c$.
\end{itemize}
Therefore the Fubini theorem implies
\begin{equation}\label{spc_integral_factor}
    \eqref{spc_integral} = \int_{T^{2y}} c_1(L_{2y})^y \int_M p_\lambda(M),
\end{equation}
so it suffices to show that the first integral equals $y!$, which we did in \cref{T2y_charclass}. The argument for $\tau_y^h$ is essentially the same, except that instead of reducing to $\int_{T^{2y}} c_1(L_{2y})^y\ne 0$, the proof reduces to the assertion that $\int_{T^{4y}} p_1(Q_{4y})^y\ne 0$, which we showed in \cref{pontryagin_cor}.
\end{proof}

\begin{lem}[Step 3]\label{lem:step 3}
    The set $B^x_{4n}$ is linearly independent in $\Omega^{\Sp^x}_{4n}\otimes\mb{Q}$ for $x\in\{c,h\}$, and $B^c_{4n+2}$ is linearly independent in $\Omega^{\Sp^c}_{4n+2}\otimes\mb{Q}$.
\end{lem}
\begin{proof}
    \cref{lem:step 1,lem:step 2} imply that elements of $\tau^x_y(\iota^x(B^r_*))$ are linearly independent for each $x,y,*$ (as long as $\Omega^{\Sp^x}_{*+y}\otimes\mb{Q}\neq 0$). By definition, we have
    \begin{subequations}
    \begin{align}
    B^x_{4n}&=\bigcup_{i=0}^n\tau^x_{2i}(\iota^x(B^r_{4(n-i)})),\\
    B^c_{4n+2}&=\bigcup_{i=0}^n\tau^c_{2i+1}(\iota^c(B^r_{4(n-i)})).
    \end{align}
    \end{subequations}
    Let $x=c$ and $\delta\in\{0,1\}$. Adapting \cref{part_n}, the proof of \cref{cor:ranks} gives us an isomorphism 
    \begin{equation}
    \begin{aligned}
    \Phi^c_{n,\delta}\colon\Omega^{\Sp^c}_{4n+2\delta}\otimes\mb{Q}&\xrightarrow{\cong}\mb{Q}\cdot\bigcup_{m=0}^n\mc{P}(m)\\
    \check{c}_1^{2(n-|\lambda|)+\delta}\otimes\check{p}_\lambda &\mapsto\left[M\mapsto\int_M c_1^{2(n-|\lambda|)+\delta}\otimes p_\lambda(M) \right]
    \end{aligned}
    \end{equation}
    of $\mb{Q}$-vector spaces. In order to verify that $B^c_{4n+2\delta}$ is a linearly independent set in $\Omega^{\Sp^c}_{4n+2\delta}$, it thus suffices to show that the $\sum_{m=0}^n P(m)\times\sum_{m=0}^n P(m)$ matrix
    \begin{align}\label{eq:dual times manifolds}
    \begin{split}
    &\bigg(\underbrace{c_1^{2n+\delta}}_{\mc{P}(0)}, 
     \underbrace{c_1^{2(n-1)+\delta}\otimes p_1}_{\mc{P}(1)},
     \ldots,
     \underbrace{
     p_n,
     \ldots,
     p_1^n}_{\mc{P}(n)}\bigg)^\intercal\cdot\\
     &\bigg(
     \underbrace{T^{4n+2\delta}}_{\mc{P}(0)}, 
     \underbrace{T^{4(n-1)+2\delta}\times\Kt^{[1]}}_{\mc{P}(1)},\ldots,\underbrace{\Kt^{[n]},\ldots,\Kt^{[1]}\times\ldots\times\Kt^{[1]}}_{\mc{P}(n)}\bigg)
    \end{split}
    \end{align}
    of characteristic numbers of elements of $B^c_{4n+2\delta}$ is invertible. We will do this by showing that the matrix in \eqref{eq:dual times manifolds} is block lower-triangular, with invertible diagonal $P(m)\times P(m)$ blocks for $m=0,\ldots,n$ (see \cref{fig:matrix}).

    \begin{figure}
    \begin{equation*}
    \bgroup\def\arraystretch{2}
        \left(
        \begin{NiceArray}{wc{2em}|wc{2em}wc{2em}|wc{2em}wc{2em}wc{2em}|wc{2em}wc{2em}wc{2em}wc{2em}}
            c_1^{2n+\delta} & \Block{1-2}{0} && \Block{1-3}{0} &&& \Block{1-4}{0} &&&\\
            \hline
            \Block{2-1}{*} & \Block{2-2}{\ddots} && \Block{2-3}{0} &&& \Block{2-4}{0} &&&\\
            &&&&&&&&&\\
            \hline
            \Block{3-1}{*} & \Block{3-2}{*} && \Block{3-3}{c_1^{2+\delta}p_{\lambda}} &&& \Block{3-4}{0} &&&\\
            &&&&&&&&&\\
            &&&&&&&&&\\
            \hline
            \Block{4-1}{*} & \Block{4-2}{*} && \Block{4-3}{*} &&& \Block{4-4}{c_1^\delta p_\lambda} &&&\\
            &&&&&&&&&\\
            &&&&&&&&&\\
            &&&&&&&&&
            \CodeAfter
            \UnderBrace[shorten,yshift=3pt]{10-1}{10-1}{\mc{P}(0)}
            \UnderBrace[shorten,yshift=3pt]{10-2}{10-3}{\mathrel{\raisebox{-1em}{$\cdots$}}}
            \UnderBrace[shorten,yshift=3pt]{10-4}{10-6}{\mc{P}(n-1)}
            \UnderBrace[shorten,yshift=3pt]{10-7}{10-10}{\mc{P}(n)}
        \end{NiceArray}
        \right)\egroup
        \vspace*{2em}
    \end{equation*}
    \caption{Characteristic numbers of $B^c_{4n+2\delta}$}\label{fig:matrix}
    \end{figure}
    
    We work with $P(i)\times P(j)$ blocks, each of which consists of the characteristic numbers
    \begin{equation}\label{eq:char number}
    \int_Mc_1^{2(n-i)+\delta}\otimes p_\lambda(M)=\int_{T^{4(n-j)+2\delta}}c_1(L_{4(n-j)+2\delta})^{2(n-i)+\delta}\int_Mp_\lambda(M),
    \end{equation}
    with $M=T^{4(n-i)+2\delta}_c\times\prod_{\ell=1}^a\Kt^{[j_\ell]}$ for $\lambda\in\mc{P}(i)$ and $(j_1,\ldots,j_a)\in\mc{P}(j)$. (This equality of characteristic numbers is given in \eqref{spc_integral}.) If $i<j$, then $2(2(n-i)+\delta)>4(n-j)+2\delta$, in which case
    \begin{equation}
    \int_{T^{4(n-j)+2\delta}}c_1(L_{4(n-j)+2\delta})^{2(n-i)+\delta}=0
    \end{equation}
    for dimension reasons. Thus \eqref{eq:char number} is 0 for $i<j$, which proves block lower-triangularity.

    For $m=0,\ldots, n$, there is a diagonal $P(m)\times P(m)$ block. By \cref{T2y_charclass} and \eqref{eq:char number}, the entries of this block take the form
    \begin{equation}
    (2(n-m)+\delta)!\int_M p_\lambda(M),
    \end{equation}
    with $M=T^{4(n-m)+2\delta}_c\times\prod_{\ell=1}^a\Kt^{[m_\ell]}$ for $\lambda=(m_1,\ldots,m_a)\in\mc{P}(m)$. The determinant of this block is non-zero if and only if the determinant of the matrix consisting of entries $\int_M p_\lambda(M)$ is non-zero. But this follows from the fact that $\iota^c(B^r_{4m})$ is linearly independent (\cref{lem:Br is basis,lem:step 1}).

    The case of $x=h$ is completely analogous. We form a matrix of characteristic numbers
    \begin{equation}
    \int_M p_1^{n-i}\otimes p_\lambda(M)=\int_{T^{4(n-i)}}p_1(Q_{4(n-i)})^{n-i}\int_Mp_\lambda(M),
    \end{equation}
    where this equality holds by the paragraph following \eqref{spc_integral_factor}. Since $\int_{T^{4(n-i)}}p_1(Q_{4(n-i)})^{n-i}$ is non-zero by \cref{pontryagin_cor}, invertibility of each $P(m)\times P(m)$ diagonal block follows from the fact that $\iota^h(B^r_{4m})$ is linearly independent (\cref{lem:Br is basis,lem:step 1}).
\end{proof}

We can now wrap up the proof of \cref{thm:main}.

\begin{proof}[Proof of \cref{thm:main}]
The statement about isomorphisms of rational bordism theories was proven in \cref{cor:bordism isomorphism}. The set $B^r_{4n}$ is a basis of $\Omega^{\Sp}_{4n}\otimes\mb{Q}$ by \cref{lem:Br is basis}. The remaining candidate bases are linearly independent by \cref{lem:step 3}, and are therefore bases by \cref{cor:ranks}.
\end{proof}

\bibliography{hyperkaehler}{}
\bibliographystyle{alpha}
\end{document}